\documentclass{article}

\usepackage{amsthm}
\usepackage{amsmath}
\usepackage{amssymb}
\usepackage{enumerate}
\usepackage{tikz}
\usepackage{tikz-3dplot}
\usetikzlibrary{positioning, arrows.meta}
\newtheorem{theorem}{Theorem}
\newtheorem{lemma}{Lemma}
\newtheorem{remark}{Remark}
\newtheorem{definition}{Definition}
\newtheorem{corollary}{Corollary}

\newtheorem{proposition}{Proposition}
\newtheorem{question}{Question}
\newtheorem{example}{Example}
\newcommand{\EE}{\mathbb{E}}

\newcommand{\ovar}{\operatorname{Var}}

\newcommand{\id}{\operatorname{id}}
\newcommand{\orthplus}{\oplus}
\begin{document}

\title{Rigid characterizations of probability measures through independence, with applications
}

\date{March 10, 2024}
\author{Thomas~A.~Courtade\\University of California, Berkeley}

\maketitle

\begin{abstract}
Three equivalent  characterizations of probability measures through  independence criteria are given.  These characterizations lead to   a family of Brascamp--Lieb-type inequalities for relative entropy, determine equilibrium states and sharp rates of convergence for certain linear Boltzmann-type dynamics, and unify an assortment of   $L^2$ inequalities in probability.    
\end{abstract}

\section{Introduction}
We start with the notation and definitions needed to state our main result.   Throughout, we work on the space $\mathbb{R}^n$, and denote the set of $n\times n$ real matrices by $M_{n\times n}(\mathbb{R})$.  Both spaces are equipped with their usual Euclidean topologies, and we let $\|\cdot\|$ denote the Euclidean norm on each space.    For a linear subspace $E\subset \mathbb{R}^n$, we let $P_E : \mathbb{R}^n \to \mathbb{R}^n$ denote the orthogonal projection of $\mathbb{R}^n$ onto $E$, represented by a matrix in  $M_{n\times n}(\mathbb{R})$.  Let $\mathbf{E}$ denote the set of all linear subspaces of $\mathbb{R}^n$.  We equip $\mathbf{E}$ with the coarsest topology such that the maps $E\in \mathbf{E} \mapsto P_E\in M_{n\times n}(\mathbb{R})$ are continuous.   For a topological space $\mathsf{X}$, we let $P(\mathsf{X})$ denote the set of Borel probability measures on $\mathsf{X}$. For our purposes, $\mathsf{X}$ will be either $\mathbf{E}$, $\mathbb{R}^n$, or a linear subspace of $\mathbb{R}^n$.   

Each linear subspace $E\subset \mathbb{R}^n$ is equipped with the induced Euclidean topology.  The orthogonal complement of $E$ with respect to the usual  inner product on $\mathbb{R}^n$ is denoted $E^{\perp}$.  For $\mu \in P(\mathbb{R}^n)$, we let $\mu_E \in P(E)$ denote the marginal   of $\mu$ on $E$.  That is, for all bounded continuous $\phi: E\to \mathbb{R}$, we have
$$
\int_{E} \phi(y) d\mu_E(y) := \int_{\mathbb{R}^n} (\phi\circ \pi_E)(x) d\mu(x),
$$
where $\pi_E: \mathbb{R}^n \to E$ is the canonical projection of $\mathbb{R}^n$ onto $E$.  Generally speaking, we write $T\#\mu$ to denote the pushforward of $\mu$ by a measurable map $T$.  For example,  $\mu_E = \pi_E \# \mu$.

 We say the measure $\mu_E$ is Gaussian if $\mu_E = N(\theta,\Sigma)$ for some $\theta\in E$ and $\Sigma\equiv \operatorname{Cov}(\mu_E)$ a symmetric nonnegative form on $E$ (in particular, we allow degenerate Gaussians).  The standard Gaussian measure on $\mathbb{R}^n$ is always  denoted by $\gamma$.

\begin{definition}
A probability measure $\mu \in P(\mathbb{R}^n)$ splits along $(E,E^{\perp})$ if it admits the product form $\mu = \mu_E\otimes \mu_{E^{\perp}}$. We say that $\mu$ splits with respect to $\xi\in P(\mathbf{E})$ if $\mu$ splits along $(E,E^{\perp})$, $\xi$-a.s. 
\end{definition}

 For $\xi \in P(\mathbf{E})$, we define   $\chi_{\xi}: \mathbb{R}^n \to \mathbb{R}$ by
$$
\chi_{\xi}(x) := \int_{\mathbf{E}} \min\{ \|P_E x\|,\|P_{E^{\perp}}x\|\} d\xi(E), ~~x\in \mathbb{R}^n.
$$
The zero-set of $\chi_{\xi}$ describes an important characteristic of $\xi$, as explained below.

\begin{proposition}\label{prop:ChiZeroSet}
The set $\chi_{\xi}^{-1}(0)\subset \mathbb{R}^n$ may be uniquely  written as the  union of mutually orthogonal, distinct linear subspaces $(E_{\alpha})_{\alpha}\subset \mathbf{E}$, which we denote  
\begin{align}
\chi_{\xi}^{-1}(0)= \cup_{\alpha} E_{\alpha}. \label{eq:indepSubspaceUnion}
\end{align}
\end{proposition}
\begin{remark}
 If the subspaces $(E_{\alpha})_{\alpha}$ are all nonzero, we do not distinguish between the cases $\left( \cup_{\alpha} E_{\alpha} \right)$ and $\{0\}\cup \left( \cup_{\alpha} E_{\alpha} \right)$ in the assertion of uniqueness.
 \end{remark}
 The proof  is deferred to Section \ref{sec:mainProof}.  Admitting it for now,   Proposition \ref{prop:ChiZeroSet} yields a canonical decomposition of $\mathbb{R}^n$ into orthogonal subspaces,   defined below.  
 
 \begin{definition}
 Fix  $\xi\in P(\mathbf{E})$.  The nonzero subspaces $(E_{\alpha})_{\alpha}$ appearing in the decomposition \eqref{eq:indepSubspaceUnion} are called the {\bf independent subspaces} of $\mathbb{R}^n$ with respect to $\xi$.   The {\bf independent decomposition} of $\mathbb{R}^n$ (with respect to $\xi$)  is the orthogonal decomposition 
 $$
\mathbb{R}^n = \left( \orthplus_{\alpha} E_{\alpha} \right) \orthplus E_{dep},
$$
where   $E_{dep}:=  ( \oplus_{\alpha} E_{\alpha}  )^{\perp}$ is the {\bf dependent subspace} of $\mathbb{R}^n$ with respect to $\xi$.  
 \end{definition}
 \begin{definition}\label{def:covSplits}
 Let $\mu \in P(\mathbb{R}^n)$, and let $V\subset \mathbb{R}^n$ be a linear subspace.  We say that the covariance of $\mu_V$ splits with respect to $\xi\in P(\mathbf{E})$ if 
 $$
\int_{\mathbf{E}} \| P_E \operatorname{Cov}(P_V \# \mu) P_{E^{\perp}} \| \, d\xi(E) = 0.
 $$
 \end{definition}
The important case to consider is when $\mu_V$ is Gaussian.  Then, for $X\sim \mu$, the vectors $P_E P_VX$ and $P_{E^{\perp}} P_VX$ are independent for $\xi$-a.e.~$E\in \mathbf{E}$.  

\subsection{Characterizations of measures through independence}

The following theorem provides rigid characterizations of measures that split with respect to a given $\xi\in P(\mathbf{E})$.  We defer the proof to Section \ref{sec:mainProof}. 
\begin{theorem}\label{thm:splittingwrtxi}
Fix $\xi\in P(\mathbf{E})$.  For $\mu\in P(\mathbb{R}^n)$, the following are equivalent:
\begin{enumerate}[1)]
\item    The measure $\mu$ splits with respect to  $\xi$.
\item In terms of the independent decomposition of $\mathbb{R}^n$ with respect to $\xi$, the measure $\mu$ admits the product form 
\begin{align}
\mu =  \left( \otimes_{\alpha} \mu_{E_{\alpha}} \right) \otimes \mu_{E_{dep}}, \label{eq:muProductXi}
\end{align}
where $\mu_{E_{dep}}$ is Gaussian with covariance that splits with respect to $\xi$. 
\item  Marginal $\mu_{E_{dep}}$ satisfies $\int_{E_{dep}} \!\!\! \log(1 + \|x\|) d\mu_{E_{dep}}\!(x)\!<\!\infty$, and $\mu$ is    the mixture 
\begin{align}
\mu = \int_{\mathbf{E}} (\mu_E \otimes  \mu_{E^{\perp}}) d\xi(E). \label{eq:MuIsMixture0}
\end{align}
\end{enumerate}
\end{theorem}

A few remarks are in order.
\begin{remark}
The marginal factors $(\mu_{E_{\alpha}})_{\alpha}$ in the decomposition \eqref{eq:muProductXi} are allowed to be arbitrary, unlike the marginal $\mu_{E_{dep}}$, which is necessarily Gaussian. 
\end{remark}
\begin{remark}
We are justified in writing the mixture \eqref{eq:MuIsMixture0} since the map 
$$\mu_E\otimes \mu_{E^{\perp}}(B) : (E,B) \to [0,1], ~~~~\mbox{Borel~}B\subset \mathbb{R}^n,
$$
is a Markov kernel.  This follows from (Borel) measurability of $(E,x,x') \mapsto 1_B(P_E x + P_{E^{\perp}}x')$ for   $x,x'\in \mathbb{R}^n$, Borel $B\subset  {\mathbb{R}^n}$ and Tonelli's theorem.
\end{remark}

 It is unclear whether the logarithmic moment assumption in the third statement of Theorem \ref{thm:splittingwrtxi} is necessary.  It may be the case that $\mu$ having the mixture form \eqref{eq:MuIsMixture0} suffices to imply that $\mu_{E_{dep}}$ satisfies the moment assumption. In view of this, the following is a natural question.
\begin{question}\label{q:moments}
Suppose there is $\theta>0$ such that $\chi_{\xi}(x) \geq \theta \|x\|$ for all $x\in \mathbb{R}^n$. If $\mu$ admits the mixture form \eqref{eq:MuIsMixture0}, does $\mu$   satisfy $\int_{\mathbb{R}^n}  \log(1 + \|x\|) d\mu(x)<\infty$?
\end{question}
An affirmative answer would allow one to remove the   moment assumption from the third statement of Theorem \ref{thm:splittingwrtxi}.  In  some simple cases that yield to   ad hoc analysis, the answer to Question \ref{q:moments} is affirmative. See Appendix \ref{app:Computations}.   

\subsection{Relationship to other work}

Theorem \ref{thm:splittingwrtxi} has connections to a variety of results appearing in the literature, many of which are discussed in detail in Section \ref{sec:Applications}.  At a more fundametal level, the equivalence 1)$\Leftrightarrow$2) is related to (and,  in fact, inspired by) Valdimarsson's characterization of extremizers \cite{valdimarsson08} in the Brascamp--Lieb inequalities  \cite{brascamp1974general, brascamp1976best,lieb1990gaussian}.  The (geometric) Brascamp--Lieb inequality asserts  that, if subspaces $(E_i)_{i=1}^k\subset \mathbf{E}$ and positive reals $(c_i)_{i=1}^k \subset (0,\infty)$ satisfy 
\begin{align}
\sum_{i=1}^k c_i P_{E_i} = \operatorname{id}_{\mathbb{R}^n}, \label{eq:frameCond}
\end{align}
then for any non-negative measurable $f_i \in L^1(E_i)$, $i=1,\dots, k$, we have
\begin{align}
\int_{\mathbb{R}^n} \prod^k_{i=1}( f_i \circ \pi_{E_i})^{c_i}(x) dx \leq   \prod_{i=1}^k \left( \int_{E_i }f_i(x_i) dx_i\right)^{c_i}.  \label{eq:BL}
\end{align}
See, for example, \cite{bennett2008brascamp}.  Admissible $(f_i)_{i=1}^k$ are said to be extremizers if they meet \eqref{eq:BL} with equality. 
Roughly speaking, Valdimarsson characterized extremal functions $(f_i)_{i=1}^k$ in \eqref{eq:BL} by defining the function 
$$
F(x,t) := \prod_{i=1}^k (H_{i,t} f_i)^{c_i}(\pi_{E_i} x), ~~~x\in \mathbb{R}^n, ~t\geq 0,
$$
where $(H_{i,t})_{t\geq 0}$ denotes the heat semigroup acting on $L^1(E_i)$, and showing that, via monotonicity along the heat flow,  $x\mapsto  F(x,t)$ must be separable in certain coordinates when the underlying $(f_i)_{i=1}^k$ are extremal.  He then uses this separability to deduce that each $f_i$ inherits similar separability properties to give a rigid characterization of extremizers.   

A probabilistic interpretation of Valdimarsson's argument becomes apparent through the lens of duality, where the Brascamp--Lieb inequalities correspond to  certain subadditivity properties of entropy \cite{carlen2009subadditivity}.    In this case, an inspection of the duality argument reveals that extremal $(f_i)_{i=1}^k$ in \eqref{eq:BL} admit interpretation as marginal densities of a single probability density on the ambient space $\mathbb{R}^n$.  A probabilistic interpretation of Valdimarsson's result that seems to have gone unnoticed, or at least has been underappreciated,   roughly manifests as the equivalence 1)$\Leftrightarrow$2) in Theorem \ref{thm:splittingwrtxi}.  Indeed, our definition of the independent decomposition of $\mathbb{R}^n$ coincides with that  defined by Valdimarsson \cite{valdimarsson08} when the probability measure $\xi$ is supported on the finite set of points $\{E_1, \dots, E_k\}$. We retain his terminology of ``independent" and ``dependent" subspaces for consistency with that literature.

The equivalence of 3) with the other two statements in Theorem \ref{thm:splittingwrtxi} does not have an analogy in the literature on Brascamp--Lieb inequalities.  However, the mixture \eqref{eq:MuIsMixture0} emerges naturally in a probabilistic context,  characterizing equilibria in certain Boltzmann-like    dynamics, as explained in Section \ref{sec:Boltzmann}.  This leads to a concrete ``physical" interpretation of the geometric Brascamp--Lieb inequalities (or, more precisely, the dual variants presented in Section \ref{sec:geomBL}) as governing convergence to equilibrium of a   stochastic process that describes an excited particle's velocity after it's  put in contact with an equilibrated bath.  

The rest of this paper is organized as follows.  Section \ref{sec:Applications} gives a sequence of applications of Theorem \ref{thm:splittingwrtxi}.  Proofs of Theorem \ref{thm:splittingwrtxi} and supporting results are found in Section \ref{sec:mainProof}.

{\bf Acknowledgement:} The author thanks Efe Aras, Pietro Caputo, and Max Fathi for stimulating conversations.  He also acknowledges NSF-CCF~1750430, the hospitality of the Laboratoire de Probabilit\'es, Statistique et Mod\'elisation (LPSM) at the Universit\'e Paris Cit\'e, and the  Invited Professor program of the Fondation Sciences Math\'ematiques de Paris (FSMP). 

\section{Applications}\label{sec:Applications}

\subsection{Bernstein-type characterization of Gaussians}\label{sec:Bernstein}
The Gaussian distribution has many surprising characterizations; for a nice overview and pointers to further references, interested readers are referred to the excellent monograph by Bryc \cite{Bryc}.   One famous characterization of the Gaussian distribution is Bernstein's characterization via independent linear forms \cite{Bernstein}.
\begin{theorem}
If $X,Y$ are independent random vectors on $\mathbb{R}^m$, such that  $(X+Y)$ and $(X-Y)$ are independent, then $X$ and $Y$ are Gaussian. 
\end{theorem}

Among its other applications, Bernstein's theorem implies Maxwell's characterization of centered isotropic Gaussian distributions as the unique rotationally-invariant distributions with independent marginals.  As a first application of Theorem \ref{thm:splittingwrtxi}, we observe that  Bernstein's theorem follows as a special case. Indeed, first note that we have:
\begin{proposition}\label{prop:GaussianCharacterization}
Let $(E_i)_{i=1}^k \subset \mathbf{E}$ satisfy $\cap_{i=1}^k (E_i \cup E_i^{\perp}) = \{0\}$.  If  $\mu\in P(\mathbb{R}^n)$ splits along $({E}_i, {E}_i^{\perp})$ for each $i=1,\dots, k$, then  $\mu$ is Gaussian.
\end{proposition}
\begin{proof}
Take $\xi$ to be a uniform distribution on $\{E_1, \dots, E_k\}$ and note that definitions imply $\chi_{\xi}(x) > 0$ for all $x\neq 0$, so that $E_{dep} = \mathbb{R}^n$.  Hence, the claim follows from the implication 1)$\Rightarrow$2) of Theorem \ref{thm:splittingwrtxi}. 
\end{proof}

To see that Bernstein's theorem follows as a consequence, consider   ambient space $\mathbb{R}^n$ with dimension $n=2m$, and define
$$
E_1 = \{ (x,0) ;  x\in \mathbb{R}^m\}, ~~~E_2 = \{ (x,x); x\in \mathbb{R}^m\},
$$ 
where $0$  denotes the zero-vector in $\mathbb{R}^m$. 
If $X,Y$ satisfy the hypotheses of Bernstein's theorem, then $\mu = \operatorname{law}(X,Y)$ splits along $(E_i,E_i^{\perp})$ for $i=1, 2$.  Of course, $\cap_{i=1}^2 (E_i \cup E_i^{\perp}) = \{0\}$, so we conclude Bernstein's theorem.

 \subsection{Geometric Brascamp--Lieb-type inequalities}\label{sec:geomBL}
  
For the applications in this section, we will assume that $\xi\in P(\mathbf{E})$  satisfies the operator inequality 
\begin{align}
\int_{\mathbf{E}} P_E d\xi(E) \leq (1-\lambda) \id_{\mathbb{R}^n} \label{eq:XiFrame}
\end{align}  
  for some $\lambda\geq 0$.  Obviously, there always exists such a $\lambda$ in the interval $[0,1]$.%

For two probability measures $P,Q$ defined on a common measurable space $(\Omega, \mathcal{F})$, the relative entropy of $P$ with respect to $Q$ is defined as 
$$
D(P\|Q) := \begin{cases}
\int_{\Omega} \log \left( \frac{dP}{dQ}\right) dP & \mbox{if $P\ll Q$}\\
+\infty & \mbox{otherwise.}
\end{cases}
$$
We briefly recall the well-known facts that $(P,Q) \mapsto D(P\|Q)$ is non-negative, convex, and weakly lower semicontinuous.  Moreover, if  $T : (\Omega, \mathcal{F}) 	\to (\Xi, \mathcal{G})$ is a measurable map between measure spaces, we have the so-called data processing inequality $D(T \#P \|T\#Q) \leq D(P\|Q)$.  If $T$ is bijective with measurable inverse, then this is an equality.  
  
 The main result of this section is the following family of (dual) Brascamp--Lieb-type inequalities.  The classical (dual, geometric) Brascamp--Lieb inequalities correspond to the special case where $\mu = \gamma$ and $\xi$ is a discrete measure supported on a finite number of subspaces.  
 \begin{theorem}\label{thm:BLForSplitMu}
If $\mu\in P(\mathbb{R}^n)$ splits with respect to  $\xi\in P(\mathbf{E})$ satisfying \eqref{eq:XiFrame}, then 
\begin{align}
\int_{\mathbf{E}} D(\nu_{E} \| \mu_{E}) d\xi(E) \leq (1-\lambda) D(\nu\|\mu), ~~~\forall \nu\in P(\mathbb{R}^n).  \label{eq:BLmuSplits}
\end{align}
\end{theorem} 
 \begin{remark}
 If $\xi$ is such that  $E_{dep}=\{0\}$, then \eqref{eq:BLmuSplits} is equivalent to Shearer's inequality for relative entropy \cite{chung1986some}.  There is no obvious way to  deduce the general case from this special case (if there were, the Brascamp--Lieb inequalities would be derivable as a special case of the Loomis--Whitney inequalities).  
  \end{remark} 
 
 \begin{proof} 
 We begin by observing that \eqref{eq:BLmuSplits} holds when the reference measure $\mu$ is replaced by the standard Gaussian measure $\gamma$, which splits along  $(E,E^{\perp})$ for every choice of subspace.  That is, we claim that 
 \begin{align}
\int_{\mathbf{E}} D(\nu_{E} \| \gamma_{E}) d\xi(E) \leq (1-\lambda) D(\nu\|\gamma), ~~~\forall \nu\in P(\mathbb{R}^n). \label{eq:GaussianBL}
\end{align}
This is a variation on the dual form of the geometric Brascamp--Lieb inequalities, which correspond to the special case where $\xi$ is a discrete measure supported on a finite number of points.  Despite this minor difference, Lehec's probabilistic proof of the geometric Brascamp--Lieb inequalities (based on a variational representation of entropy due to F\"ollmer \cite{Follmer85, Follmer86}) can be executed nearly verbatim to conclude \eqref{eq:GaussianBL}; see \cite[proof of Theorem 16]{lehecBL}.

Now, we'll use a transportation argument to conclude \eqref{eq:BLmuSplits} by leveraging the rigid product structure \eqref{eq:muProductXi} implied by the splitting property of $\mu$.  Toward this end, let $T: \mathbb{R}^n \to \mathbb{R}^n$ be a homeomorphism satisfying \begin{align}
T \circ P_E = P_E \circ T,  \mbox{ $\xi$-a.s.} \label{eq:TcommutesPE}
\end{align}
As a homeomorphism, $T^{-1}$ exists and is continuous (and therefore measurable), and we have  for $\xi$-a.e. $E$
\begin{align*}
D(  \nu_E \|  (T\#\gamma)_E )  
& = D((P_E\circ T)\# (T^{-1}\#\nu) \| (P_E\circ T) \#  \gamma )\\
&= D((T\circ P_E )\# (T^{-1}\#\nu) \| (T\circ P_E) \#  \gamma )\\
&= D(P_E \# (T^{-1}\#\nu) \|  P_E  \#  \gamma ) = D(  (T^{-1}\#\nu)_E \|     \gamma_E ),
\end{align*}
where the first line is definitions, the second is the assumed commutativity, and the third follows since relative entropy is invariant under homeomorphisms of the measure space.  Thus, integrating with respect to $\xi$ gives
\begin{align*}
\int_{\mathbf{E}} D(  \nu_E \|  (T\#\gamma)_E )   d\xi(E) 
&=  \int_{\mathbf{E}}D(  (T^{-1}\#\nu)_E \|     \gamma_E )  d\xi(E) \\
&\leq (1-\lambda) D( T^{-1}\#\nu \|     \gamma ) \\
&= (1-\lambda) D(  \nu \|    T\#\gamma ),
\end{align*}
where the inequality is \eqref{eq:GaussianBL} applied to the measure $\nu \leftarrow T^{-1}\#\nu$, and the final step is again because relative entropy is invariant under homeomorphisms of the measure space.  Hence, \eqref{eq:BLmuSplits} holds whenever $\mu$ is the pushforward of $\gamma$ by a homeomorphism  $T: \mathbb{R}^n \to \mathbb{R}^n$ that commutes with $P_E$,   $\xi$-a.s.   

Now, let $\mu$ split with respect to $\xi$.  By Theorem \ref{thm:splittingwrtxi},  $\mu$ admits the product form
$$
\mu =  \left( \otimes_{\alpha} \mu_{E_{\alpha}} \right) \otimes \mu_{E_{dep}},
$$
where $\mu_{E_{dep}} \in P(E_{dep})$ is Gaussian, and $\mu_{E_{\alpha}} \in P(E_{\alpha})$.  We may assume without loss of generality that $\mu_{E_{dep}}$ is centered, and therefore $\mu_{E_{dep}} = N(0,\Sigma)$, where the covariance $\Sigma : E_{dep} \to E_{dep}$ splits with respect to $\xi$.  In particular, 
$$
\Sigma^{1/2} (E_{dep}\cap E) \subset (E_{dep}\cap E) ~~\mbox{and}~~\Sigma^{1/2}  (E_{dep}\cap E^{\perp}) \subset (E_{dep}\cap E^{\perp}), ~~~\xi-a.s.,
$$
due to the splitting property of $\mu$, and the fact that $\Sigma$ and its positive semidefinite square root $\Sigma^{1/2}$ share common eigenspaces.  

Now, let us assume $\mu$ has smooth non-vanishing density with respect to Lebesgue measure.  For each nonzero independent subspace $E_{\alpha}$, there exists a homeomorphism $T_{\alpha}: E_{\alpha}\to E_{\alpha}$ such that $T_{\alpha} \#\gamma_{E_{\alpha}} = \mu_{E_{\alpha}}$ \cite[Corollary 1]{CorderoFigalli}.  On $E_{dep}$ we define the linear map $T_{dep}: x\in E_{dep} \mapsto  \Sigma^{1/2}x$; therefore $T_{dep}\#\gamma_{E_{dep}} = \mu_{E_{dep}}$.  Since we are assuming $\mu$ has full support, $\Sigma$ is positive definite, and therefore $T_{dep}$ is also a homeomorphism.  Now, we define the operator $T:\mathbb{R}^n \to \mathbb{R}^n$ via the direct sum 
$$
T = \left( \oplus_{\alpha} T_{\alpha} \right) \oplus T_{dep},
$$
with individual components acting on the corresponding components of the independent decomposition of $\mathbb{R}^n$. 
By construction,   $T$ is a homeomorphism satisfying \eqref{eq:TcommutesPE}.   Hence,   \eqref{eq:BLmuSplits}  holds whenever $\mu$ has smooth non-vanishing density with respect to Lebesgue measure.

The general case where $\mu$ fails to admit a non-vanishing density is handled by localization.  To this end,  let $(H^*_t)_{t\geq 0}$ denote the (adjoint) heat semigroup acting on $P(\mathbb{R}^n)$.  Since $H^*_t$ operates independently on orthogonal components, it is a  consequence of the data processing inequality that 
\begin{align*}
D(\nu_E \| \mu_E) &\geq D((H^*_t \nu)_E \| (H^*_t  \mu)_E).
\end{align*}
By weak lower semicontinuity of entropy, we have the reverse inequality in the limit as $t\downarrow 0$:
 $$
 \liminf_{t\geq 0} D((H^*_t \nu)_E \| (H^*_t \mu)_E) \geq D(\nu_E\|\mu_E). 
 $$

For every $t>0$ the measure $H^*_t \mu$ admits a smooth non-vanishing density with respect to Lebesgue measure.  Since $\mu \mapsto H^*_t\mu$ preserves the splitting property of $\mu$ for all $t\geq 0$, we have   
\begin{align*}
\int_{\mathbf{E}} D(\nu_E \| \mu_E) d\xi(E) 
&\leq \liminf_{t \to 0} \int_{\mathbf{E}} D((H^*_t  \nu)_E \| (H^*_t  \mu)_E) d\xi(E)  \\
&\leq  \liminf_{t\to 0}  (1-\lambda) D( H^*_t  \nu \| H^*_t \mu)\\
&\leq  (1-\lambda) D(   \nu \|  \mu),
\end{align*} 
 where the first inequality is Fatou's lemma, the second is \eqref{eq:BLmuSplits} for reference measure having smooth non-vanishing density, and the third inequality is data processing.  
\end{proof}

\subsection{Linear Boltzmann-type dynamics} \label{sec:Boltzmann}

Consider a physical experiment where two particles of equal mass and respective velocities $v, v_* \in \mathbb{R}^n$ undergo an elastic collision. By conservation of energy and momentum, the particles necessarily retain velocity components on some subspace $E\in \mathbf{E}$, and exchange on $E^{\perp}$.  That is, the post-collision velocities of the first and second particles are, respectively:
\begin{align}
v' = P_E v + P_{E^{\perp}}v_*  , ~~\mbox{and}~~v'_* = P_E v_* + P_{E^{\perp}}v. \label{model}
\end{align}
In a Hamiltonian-based model, the subspace $E$ would be a function of the particle positions and a suitable potential energy function.  In many-particle systems, however,  it will  generally be more tractable to assume a spatially homogenous model where the subspace $E$ is drawn randomly, according to a given distribution $\xi\in P(\mathbf{E})$.  Then, the post-collision velocities are random variables, with distributions induced by $\xi$ and the pre-collision velocities $v,v_*$.  We remark that the post-collision kinetic energy of the (first) particle can be written as $\frac{1}{2}\|v'\|^2 = \frac{1}{2}\|P_Ev\|^2 + \frac{1}{2}\|P_{E^{\perp}}v_*\|^2$. So, when $E\sim \xi$, inequality \eqref{eq:XiFrame} has the interpretation that, on average, a fraction at least $\lambda$ of the total kinetic energy is exchanged  through   collision.

Switching gears for a moment, let $\mu\in P(\mathbb{R}^n)$, and define a Markov semigroup $P_t= e^{t\mathcal{L}}, t\geq 0,$ via the infinitesimal generator
$$
\mathcal{L} f(v) := \int_{\mathbf{E}} \int_{ \mathbb{R}^n} f(P_E v + P_{E^{\perp}}v_*) d\mu(v_*) d\xi(E) - f(v), 
$$
on bounded measurable $f: \mathbb{R}^n \to \mathbb{R}$.  For initial data $\nu_0\in P(\mathbb{R}^n)$, we define the evolution $(\nu_t)_{t\geq 0}$ via duality
$$
\int_{\mathbb{R}^n} f d\nu_t := \int_{\mathbb{R}^n} P_t f d\nu_0 , ~~~f\in C_b(\mathbb{R}^n). 
$$

The Markov semigroup $(P_t)_{t\geq 0}$ is related to the physical model previously described as follows.  Consider an experiment where a unit-mass particle, excited with initial velocity $V_0 \sim \nu_0$, is placed in contact with a {bath} consisting of unit-mass particles, each having velocities distributed i.i.d.~according to some background distribution $\mu$.  If elastic collisions modeled by \eqref{model} occur between   the excited particle and independent particles in the bath following a rate-1 Poisson point process, with subspace $E\sim \xi$ drawn independently for each collision, the velocity $V_t$ of the excited particle at time $t$  has law $\nu_t$.  The bath is in  {equilibrium} with respect to the described dynamics if, for $\nu_0 = \mu$, we have $\nu_t = \mu$ for all $t\geq 0$.  The \emph{temperature} of the bath is $\int_{\mathbb{R}^n}\|v\|^2 d\mu(v)$.

Theorem \ref{thm:splittingwrtxi}  immediately provides a rigid characterization of finite-temperature\footnote{The assumption of finite temperature is stronger than required, since our application of Theorem \ref{thm:splittingwrtxi} only requires $\mu_{E_{dep}}$ to have finite logarithmic moment.  However, all realizable physical systems have finite temperature, so it seems more physically relevant to impose this assumption instead of merely finite  logarithmic moments on $E_{dep}$.} baths in equilibrium with respect to our dynamics.
\begin{proposition}\label{prop:EquilibriumSplits}
A finite-temperature bath with background distribution $\mu$ is in equilibrium with respect to the dynamics $(P_t)_{t\geq 0}$ if and only if $\mu$ splits with respect to $\xi$.   Moreover, in this case, $\mu$ is reversible for $(P_t)_{t\geq 0}$.
\end{proposition}
\begin{proof}
In order for the bath to be invariant to the described dynamics, we require that $\mu$ is an invariant measure of the semigroup $(P_t)_{t\geq 0}$.  In particular, for all $f\in C_b(\mathbb{R}^n)$, we have $\int_{\mathbb{R}^n} \mathcal{L} f d \mu = 0$.  This may be rewritten as 
\begin{align*}
\int_{\mathbf{E}} \int_{\mathbb{R}^n}  f  d(\mu_E \otimes  \mu_{E^{\perp}} )  d\xi(E) = \int_{\mathbf{E}} \int_{\mathbb{R}^n}  f d \mu d\xi(E) = \int_{\mathbb{R}^n}  f d \mu.
\end{align*}
In other words, the bath is in equilibrium if and only if $\mu$ is the mixture \eqref{eq:MuIsMixture0}. Under the finite-temperature assumption, this is equivalent to $\mu$ splitting with respect to $\xi$  by Theorem \ref{thm:splittingwrtxi}.

To show reversibility, we need to show $\int_{\mathbb{R}^n} f \mathcal{L}g d\mu = \int_{\mathbb{R}^n} g \mathcal{L}f d\mu$ for, say, bounded continuous $f,g$.  This is easily verified using the characterization that $\mu$ splits along $(E,E^{\perp})$, $\xi$-a.s. 
\end{proof}
 
 Having characterized all equilibrium states for our dynamics, we now establish quantitative rates of convergence to equilibrium.  

\begin{theorem}\label{thm:LSI_PI} Let $\lambda\geq 0$ be such that  \eqref{eq:XiFrame} holds.  If $\mu$ splits with respect to  $\xi$, then 
\begin{align}
D(\nu_t  \| \mu ) \leq e^{-\lambda t} D(\nu_0 \| \mu), ~~~\forall t\geq 0,  ~\nu_0 \in P(\mathbb{R}^n).  \label{eq:EntropyDecay}
\end{align}
Additionally, for all $f\in L^2(\mu)$, 
\begin{align}
\ovar_{\mu}(P_t f)  \leq e^{-2\lambda t} \ovar_{\mu}(f). \label{eq:specGapIneq}
\end{align}
\end{theorem}
The theorem is stated in such a way that no moment assumptions on $\mu$ are imposed.  However, in view of Proposition \ref{prop:EquilibriumSplits}, the physical interpretation of Theorem \ref{thm:LSI_PI} is that it governs  convergence toward equilibrium for the process   that results from placing an excited particle in contact with a finite-temperature equilibrated bath.   In view of this interpretation, if the best choice of $\lambda$ in \eqref{eq:XiFrame} is $\lambda = 0$, then our particles do not interact in some direction, and convergence to equilibrium cannot be guaranteed.  However, if $0 < \lambda \leq 1$, then \eqref{eq:EntropyDecay} implies a strict trend to  equilibrium  $\mu$.   The  respective  convergence rates of $\lambda$ and $2\lambda$   in \eqref{eq:EntropyDecay} and \eqref{eq:specGapIneq} are sharp; an explanation follows the proof.

 \begin{remark}
The entropy-decay estimate \eqref{eq:EntropyDecay} is equivalent to a log-Sobolev inequality with constant $2 / \lambda$ \cite[Theorem 5.2.1]{BGL}, and the variance estimate \eqref{eq:specGapIneq}  is equivalent to  a Poincar\'e  inequality with constant $1/\lambda$ \cite[Theorem 4.2.5]{BGL}.   
\end{remark}

 \begin{proof}[Proof of Theorem \ref{thm:LSI_PI}]
We'll prove the entropy decay estimate \eqref{eq:EntropyDecay} first; we can assume without loss of generality that $D(\nu_0\|\mu) < \infty$, else there is nothing to prove.  Toward this end, note that the Poisson timing of the collision process implies
\begin{align}
\nu_t = e^{-t} \nu_0 + t e^{-t} \int_{\mathbf{E}} \nu_{0 E}\otimes \mu_{E^{\perp}}d\xi(E) + (1- (1+t)e^{-t}) \tilde{\nu}_t, \label{eq:PoissonMixture}
\end{align}
where $\tilde{\nu}_t \in P(\mathbb{R}^n)$   is the conditional law of the excited particle's velocity at time $t$, given that two or more collisions occurred before time $t$. Thus, for all $t\geq 0$ we have
\begin{align*}
& D(\nu_t\|\mu) - D(\nu_0\|\mu) \\
&\leq  (e^{-t}-1) D(\nu_0\|\mu) + t e^{-t} \int_{\mathbf{E}} D(\nu_{0E}\|\mu_E)  d\xi(E) + (1- (1+t)e^{-t}) D(\tilde{\nu}_t\|\mu)   \\
&\leq (e^{-t}-1) D(\nu_0\|\mu) + t e^{-t} (1-\lambda) D(\nu_{0}\|\mu)   + (1- (1+t)e^{-t}) D( {\nu}_0\|\mu)   \\
&=  -t e^{-t}\lambda  D(\nu_0\|\mu),
\end{align*}
where the first inequality  follows by convexity of relative entropy applied to the mixture  \eqref{eq:PoissonMixture}, and the second inequality follows from an application of \eqref{thm:BLForSplitMu} and the data processing inequality for relative entropy which implies $D(\tilde{\nu}_t\|\mu)\leq D(\nu_0 \| \mu)$ (this uses the fact that $\mu$ is invariant under our dynamics, so that if $\nu_0 = \mu$, then $\tilde{\nu}_t = \mu$).  

Hence, from the above and the semigroup property, we have the estimate 
\begin{align*}
 \frac{d}{dt}^+ D(\nu_t\|\mu) \leq  -\lambda  D(\nu_t\|\mu), ~~~t\geq 0.
\end{align*}
An application of  Gr\"onwall's lemma gives \eqref{eq:EntropyDecay}.

Inequality \eqref{eq:EntropyDecay} can be linearized to obtain the variance decay inequality \eqref{eq:specGapIneq}, albeit with the suboptimal rate $\lambda$ in the exponent.  The improved rate of $2\lambda$ will be obtained by linearizing \eqref{eq:BLmuSplits} directly.  Toward this end, recall that the relative entropy of $P\ll Q$ can be written as
$$
D(P\|Q) = \int \frac{dP}{dQ}\log\left(\frac{dP}{dQ} \right)dQ.
$$
Therefore, if $P$ is a perturbation of $Q$ in the sense that $dP = (1+\epsilon f) dQ$ for a bounded measurable function $f$ and $\epsilon$ sufficiently small, then   Taylor expansion of $x\in \mathbb{R}^+ \mapsto x \log x$ about $x=1$ gives the local behavior of relative entropy
$$
D(P\|Q) = \frac{\epsilon^2}{2} \ovar_Q(f) +o(\epsilon^2),
$$
where the first-order term is absent since $f$ necessarily satisfies $\int f dQ = 0$ for $P$ to be a probability measure.  By Taylor's theorem, we remark that the little-$o$ term has modulus at most 
$$
\frac{(\epsilon \|f\|_{\infty})^3 }{6 (1-\epsilon \|f\|_{\infty})^2}.
$$
So, consider a bounded measurable  function $f: \mathbb{R}^n\to \mathbb{R}$ satisfying $\int f d\mu= 0$ and define $d\mu^{\epsilon}:= (1+\epsilon f) d\mu$, which is a valid probability measure for all $\epsilon$ sufficiently small.  Let $X\sim \mu$.  For $E\in \mathbf{E}$, definitions imply 
$$
d\mu^{\epsilon}_E = (1 + \epsilon \EE[f(X) |P_E X]) d\mu_E, 
$$
where $\EE[f(X) |P_E X]$ is the conditional expectation of $f(X)$ under $\mu$ with respect to the $\sigma$-algebra generated by $P_E X$.  
Thus, by linearization and \eqref{eq:BLmuSplits}, we have
\begin{align*}
\frac{\epsilon^2}{2} \int_{\mathbf{E}} \ovar( \EE[f(X)|P_E X] )d \xi(E) + o(\epsilon^2) &= 
\int_{\mathbf{E}}  D\left( \mu^{\epsilon}_E  \big\| \mu_E \right)d \xi(E) \\ 
&\leq   (1-\lambda) D(\mu^{\epsilon} \| \mu) \\
&= (1-\lambda) \frac{\epsilon^2}{2}\ovar( f(X)  )+ o(\epsilon^2),
\end{align*}
where all variances and expectations are with respect to $\mu$.
Dividing through by  $\epsilon^2$ and letting $\epsilon$ vanish establishes the inequality 
\begin{align}
 \int_{\mathbf{E}} \ovar( \EE[f(X)|P_E X] )d \xi(E) \leq  (1-\lambda) \ovar( f(X)  ) \label{eq:fVarDrop}
\end{align}
for all bounded measurable $f$ with $\int f d\mu = 0$.  Inequality \eqref{eq:fVarDrop} is invariant to replacing $f$ by $f + c$ for $c\in \mathbb{R}$, so the zero-mean assumption can be eliminated.  Thus, a standard density argument allows us to conclude the same holds for all $f\in L^2(\mu)$.   Finally, using the classical variance decomposition 
$$\ovar( f(X) )  =  \EE[\ovar(f(X)|P_E X)] +   \ovar(\EE[f(X)|P_E X]),$$
 we may rewrite \eqref{eq:fVarDrop} in the form 
\begin{align}
\ovar_{\mu}( f  )  = \ovar( f(X)  )  &\leq \frac{1}{\lambda}  \int_{\mathbf{E}} \EE[ \ovar( f(X)|P_E X  )] d \xi(E) \label{eq:fVarDropEfronStein}\\
&=\frac{1}{2 \lambda} \int_{\mathbb{R}^n}  \left( \mathcal{L}(f^2) -   2 f \mathcal{L}(  f) \right) d\mu. \notag
\end{align}
By \cite[Theorem 4.2.5]{BGL}, this is equivalent to \eqref{eq:specGapIneq}.
 \end{proof}

With the proof complete, we now address  optimality of the rates of convergence in \eqref{eq:EntropyDecay} and \eqref{eq:specGapIneq}.  To start, let  $\xi\in P(\mathbf{E})$ and $\lambda\geq 0$ satisfy 
\begin{align}
  \int_{\mathbf{E}}    P_E   d\xi(E)  = (1-\lambda) \id_{\mathbb{R}^n}. \label{frameEqual}
\end{align}
Now, take $\mu = \gamma$, and $\nu_0 = N(\theta, \id_{\mathbb{R}^n})$ for fixed $\theta\in \mathbb{R}^n$.  By construction, $\nu_t$ is the mixture  
\begin{align*}
\nu_t &= \sum_{k\geq 0} \frac{t^k e^{-t}}{k!}   \int_{\mathbf{E}^k}  N(P_{E_1}  \cdots P_{E_k} \theta, \id_{\mathbb{R}^n}) d\xi^{\otimes k}(E_1,\dots, E_k) . 
\end{align*}
So, using the  Donsker--Varadhan variational formula for entropy in combination with the identity (due to \eqref{frameEqual})
$$
 \int_{\mathbf{E}^k}  \|P_{E_1} P_{E_2} \dots P_{E_k}\theta\|^2 d\xi^{\otimes k}(E_1, E_2,\dots, E_k) = (1-\lambda)^k \|\theta\|^2,~~k\geq 1,
$$
we have  for any $\beta >1$, the crude lower bound
\begin{align*}
D(\nu_t \|\gamma) &\geq \int_{\mathbb{R}^n} \frac{1}{2\beta}|x|^2 d\nu_t(x) - \log \left( \int_{\mathbb{R}^n}  e^{\frac{|x|^2}{2\beta}} d\gamma(x) \right) \\
&= \frac{n}{2\beta} - \frac{n}{2}\log\left( \frac{\beta}{\beta-1}\right)  + \frac{1}{\beta} e^{-\lambda t} \underbrace{D(\nu_0 \|\gamma)}_{\frac{1}{2}\|\theta\|^2}.
\end{align*}
Hence, for any choice of  $t_0 \geq 0$ and $\epsilon >0$, we can take $\beta$ sufficiently close to 1 and $\theta$ with sufficiently large norm such that 
$$
D(\nu_{t_0} \|\gamma) \geq  (1-\epsilon)  e^{-\lambda t_0} D(\nu_0 \|\gamma).
$$
By an application of \eqref{eq:EntropyDecay} and the semigroup property, this means we must have 
$$
D(\nu_{t} \|\gamma) \geq  (1-\epsilon)  e^{-\lambda t} D(\nu_0 \|\gamma), ~~~\forall t \in [0, t_0].
$$
It follows that the rate $\lambda$ in \eqref{eq:EntropyDecay} cannot be improved when $\mu=\gamma$.  Using the transport map construction in the proof of Theorem \ref{thm:BLForSplitMu}, this example with  reference measure equal to $\gamma$ can be transported to any reference measure $\mu$ that splits with respect to $\xi$.  Therefore, the rate $\lambda$ cannot be improved in general.

To see that the rate in \eqref{eq:specGapIneq} is sharp, assume $\mu$ splits with respect to $\xi$, and assume for simplicity that it has finite second moments.  Put $f(x) = u\cdot x$, where $u$ is an eigenvector of $\int P_E d\xi(E)$ with eigenvalue $1-\lambda$.  That is, 
\begin{align}
\left(  \int_{\mathbf{E}}    P_E   d\xi(E) \right) u = (1-\lambda) u. \label{eq:uEigenVector}
\end{align}
For   $X\sim \mu$, we have
 \begin{align*}
  \int_{\mathbf{E}} \ovar( \EE[f(X)|P_E X] )d \xi(E) &= \int_{\mathbf{E}} \EE \| u \cdot P_E X\|^2 d\xi(E)\\
  &=  \int _{\mathbf{E}} u^T P_E \operatorname{Cov}(\mu) P_E u \, d\xi(E) \\
  &= u^T   \operatorname{Cov}(\mu)  \left( \int _{\mathbf{E}}  P_E d\xi(E)  \right) u \\
  &=  (1-\lambda) u^T   \operatorname{Cov}(\mu)   u \\
  &=  (1-\lambda) \ovar( f(X) ) ,
\end{align*}
where we used that $P_E$ and $\operatorname{Cov}(\mu)$ commute for $\xi$-a.e.~$E\in \mathbf{E}$ (since $\mu$ splits with respect to $\xi$), and \eqref{eq:uEigenVector}.   Hence, the constant $(1-\lambda)$ is best-possible in the inequality \eqref{eq:fVarDrop}, and this is equivalent to \eqref{eq:specGapIneq}.

\subsubsection*{Remarks on related literature}

The physical model proposed in this Section can be regarded as spatially homogeneous linear Boltzmann dynamics, where the collision kernel is replaced by the random choice of subspace dictated by the distribution $\xi$.  In this respect, the interactions considered here are more similar in spirit to Kac's simplified model \cite{Kac} than the full   theory of the linear Boltzmann equation.  The Brascamp--Lieb inequalities on the sphere  have been interpreted as governing convergence to equilibrium in Kac's model \cite{carlen2004sharp}, but we are unaware of any previous kinetic interpretation of the Euclidean Brascamp--Lieb inequalities.   In a related spirit, there is a line of recent work by P.~Caputo and collaborators that show rapid mixing of certain Markov processes is closely connected to validity of approximate Shearer-type inequalities for entropy  (see, e.g., \cite{CaputoMenzParisi, BlancaEtAl, BristielCaputo, CaputoParisi, CaputoSinclair}).  It is an interesting question as to whether these results fit within the context of (yet-to-be-explored) Brascamp--Lieb-type inequalities on the discrete spaces they consider.

\subsection{Inequalities in probability}

It's well-known that the Brascamp--Lieb inequalities \eqref{eq:BL} imply many classical analytic and geometric inequalities (e.g., the H\"older, sharp Young, and Loomis--Whitney inequalities).  None of these applications  require Valdimarsson's characterization of extremizers.  However, by incorporating the latter circle of ideas into the picture through their manifestation in Theorem \ref{thm:splittingwrtxi}, we eventually arrived at the  probabilistic inequalities \eqref{eq:fVarDrop} and \eqref{eq:fVarDropEfronStein}, which we summarize below.   In this Section, we'll explain how they generalize a variety of  familiar inequalities in probability. 
\begin{theorem}\label{thm:LinearizedBL}
Let $\lambda\geq 0$ be such that  \eqref{eq:XiFrame} holds, and let $\mu$ split with respect to $\xi$.  If $X\sim \mu$, then for all $f\in L^2(\mu)$ 
\begin{align}
 \int_{\mathbf{E}} \ovar( \EE[f(X)|P_E X] )d \xi(E) \leq  (1-\lambda) \ovar( f(X)  ) \label{eq:fVarDrop2}
\end{align}
and, equivalently, 
\begin{align}
 \ovar( f(X)  )  &\leq \frac{1}{\lambda}  \int_{\mathbf{E}} \EE[ \ovar( f(X)|P_E X  )] d \xi(E). \label{eq:fVarDropEfronStein2}
 \end{align}
 \end{theorem}
Since \eqref{eq:fVarDrop2} and \eqref{eq:fVarDropEfronStein2} follow from linearizing the Brascamp--Lieb-type inequalities \eqref{eq:BLmuSplits}, we find it appropriate to refer to them as \emph{linearized Brascamp--Lieb inequalities}.  They were described by the author in the extended abstract \cite{CourtadeIZS24}, along with the following few examples which we repeat here (see also \cite{Courtade21}, in the context of linearizing Shearer's inequality).

Since the standard Gaussian measure $\gamma$ splits with respect to any $\xi\in P(\mathbf{E})$, we immediately obtain a family of  variance inequalities for $\gamma$. %
\begin{example}
If $X\sim \gamma$, then  \eqref{eq:fVarDrop2} and   \eqref{eq:fVarDropEfronStein2} hold for every $\xi \in P(\mathbf{E})$ and $\lambda\geq 0$ satisfying \eqref{eq:XiFrame}.
\end{example}

In general, if one starts with a  random vector $X$ with given independence structure, it may be possible to construct a measure $\xi\in P(\mathbf{E})$ such that the hypotheses of Theorem \ref{thm:LinearizedBL} are satisfied.  This approach handily recovers an assortment of classical $L^2$ inequalities, as demonstrated in the following few examples.

\begin{example}(Efron--Stein inequality \cite{EfronStein, steele})
Let $X = (X_i)_{i=1}^k$ be a random vector with independent components $(X_i)_{i=1}^k$, and define 
$$X^{(i)} =   (X_1, \dots, X_{i-1}, X_{i+1}, \dots X_k).$$
  For any measurable $f$ with $\ovar(f(X))<\infty$,
\begin{align}
 \ovar(f(X)) \leq \sum_{i=1}^k   \EE[ \ovar(f(X)|X^{(i)} )] . \label{eq:EfronStein}
\end{align}
\end{example}
\begin{proof}
We can assume $X$ takes values in $\mathbb{R}^n$, and choose  $E_i$  such that $X_i$ is the component of $X$ in $E_i^{\perp}$.  This implies the orthogonal decomposition $\mathbb{R}^n = {\orthplus}_{1\leq i \leq k } E^{\perp}_i$, which yields  
$$
\frac{1}{k-1} \sum_{i=1}^k P_{E_i} = \operatorname{id}_{\mathbb{R}^n}.
$$
By the independence hypothesis, the law of $X$ splits along $(E_i, E_i^{\perp})$ for each $i=1,\dots, k$, and therefore \eqref{eq:EfronStein} follows from \eqref{eq:fVarDropEfronStein2} by taking $\xi$ equal to the uniform distribution on   $\{E_1, \dots, E_k\}\subset \mathbf{E}$. 
\end{proof}
We remark that the Efron--Stein inequality is often interpreted as a spectral gap inequality for the Gibbs sampling procedure where coordinates of a vector $X$ are resampled, uniformly at random.  From this perspective, \eqref{eq:fVarDropEfronStein2} may be thought of as the same, except where the component of  $X$ lying in subspace $E^{\perp}$ is resampled according to the distribution $E\sim \xi$.

\begin{example}(Dembo--Kagan--Shepp inequality \cite{dembo2001remarks}) Let $(X_i)_{i\geq 1}$ be a sequence of i.i.d.~random vectors, and define the cumulative sum $S_n = \sum_{j=1}^n X_j$.  For any measurable $g$, 
\begin{align}
\ovar(\EE[g(S_n)|S_m]) \leq \frac{m}{n} \ovar(g(S_n)), ~~~~n\geq m\geq 1. \label{eq:DKS}
\end{align}
\end{example} 
\begin{proof}
For simplicity of notation, we'll assume each $X_i$ is one-dimensional.   Consider the random vector $X = (X_1, \dots, X_n)$ taking values in  $\mathbb{R}^n$, with $X_j$ the projection of $X$ along natural basis vector $e_j$, $j=1,\dots, n$. Take $(E_i)_{i=1}^k$ be an enumeration of all $k = {n \choose m}$ subspaces of $\mathbb{R}^n$, each equal to the linear span of exactly $m$ natural basis vectors.  By construction, the law of $X$ splits along $(E_i, E_i^{\perp})$ for each $i=1,\dots,k$, and 
$$
 \sum_{i=1}^k P_{E_i} \xi(\{E_i\})= \frac{m}{n}\operatorname{id}_{\mathbb{R}^n}.
$$
for $\xi$ the uniform measure on the discrete set $\{E_1, \dots, E_k\}$.   By symmetry, $\EE[ g(S_n) | P_{E_i} X]$ are equal in law for each $i=1, \dots, k$.  So, an application of \eqref{eq:fVarDrop2} with $f(X) = g(S_n)$ gives 
\begin{align*}
\ovar(\EE[g(S_n)|X_1, \dots, X_m]) \leq \frac{m}{n} \ovar(g(S_n)), ~~~~n\geq m\geq 1.  
\end{align*}
The claim   follows since $S_m$ is a sufficient statistic of $(X_1, \dots, X_m)$ for $S_n$. 
\end{proof}
 \begin{remark}
 In the special case where $(X_i)_{i\geq 1}$ are i.i.d.~$\gamma$, inequality \eqref{eq:DKS} yields the  optimal  estimate for contraction of the Ornstein--Uhlenbeck semigroup in $L^2(\gamma)$.  This is equivalent to the Gaussian Poincar\'e inequality.
 \end{remark}
 
 A straightforward Corollary of Theorem  Theorem \ref{thm:LinearizedBL} is the following improvement of Jensen's inequality by a factor of $(1-\lambda)$, under the structural assumption that $\mu$ splits with respect to $\xi$.
\begin{corollary} 
Let $\mu$ split with respect to $\xi$, and  $\lambda\geq 0$ satisfy  \eqref{eq:XiFrame}.    If $(\psi_E)_{E\in \mathbf{E}}$ is a collection of functions from $\mathbb{R}^n$ to $\mathbb{R}$ such that 
$$(E,x)\in \mathbf{E}\times \mathbb{R}^n  \mapsto (\psi_E \circ P_E)(x) \in \mathbb{R}$$
 is   measurable, then 
 \begin{align}
\ovar_{\mu}\left( \int_{\mathbf{E}} (\psi_E\circ P_E)  d\xi(E)  \right)  \leq  (1-\lambda)  \int_{\mathbf{E}}  \ovar_{\mu}\left(  \psi_E\circ P_E   \right) d\xi(E) .\label{eq:BLvarianceDrop}
 \end{align}
\end{corollary}
\begin{proof}  Put $f(x) := \int_{\mathbf{E}} (\psi_E\circ P_E)( x) d\xi(E)$, and let $X\sim \mu$.  Applying Cauchy-Schwarz twice followed by \eqref{eq:fVarDrop2}, we have
\begin{align*}
&\ovar(f(X)) \\
&=%
\int_{\mathbf{E}} \operatorname{Cov}( f(X) , (\psi_E\circ P_E)(X) ) d\xi(E)  \\
 &= 
\int_{\mathbf{E}}  \operatorname{Cov}( \EE[f(X)|P_E X] , (\psi_E\circ P_E)(X) )  d\xi(E) \\
 &\leq    \int_{\mathbf{E}}  \ovar \left( \EE[f(X)|P_E X]  \right)^{1/2} \ovar \left((\psi_E\circ P_E)(X)  \right)^{1/2}     d\xi(E) \\
 &\leq \left(   \int_{\mathbf{E}}  \ovar \left(  \EE[ f(X)|P_E X]  \right)  d\xi(E)  \right)^{1/2}  \left(    \int_{\mathbf{E}}  \ovar \left( (\psi_E\circ P_E)(X) \right)  d\xi(E)  \right)^{1/2}  \\
 &\leq   (1-\lambda)^{1/2} \ovar\left( f(X)  \right)^{1/2} \left(    \int_{\mathbf{E}}  \ovar \left( (\psi_E\circ P_E)(X) \right)  d\xi(E)  \right)^{1/2} .
\end{align*} 
\end{proof}

As an example,  we recover an  inequality  due to Madiman and Barron \cite{madiman2007generalized}, which is itself   a generalization of a classical result on $U$-statistics due to Hoeffding \cite{hoeffding1948class}.    To state it, recall that $\mathcal{T}\subset 2^{[n]}$ is called an $r$-cover of $[n]:=\{1,\dots, n\}$ if each element of $[n]$ is contained in at most $r$ members of $\mathcal{T}$. 
\begin{example}[Madiman--Barron   inequality] Let $X = (X_m)_{m=1}^n$ be a collection of $n$ independent random random variables,  let $(S_i)_{i=1}^k \subset 2^{[n]}$  be an $r$-cover of $[n]$, and define $X_{S_i} := (X_m)_{m\in S_i}$.  For any real-valued measurable $\psi_i : X_{S_i} \mapsto \psi_i(X_{S_i})$, 
 \begin{align}
\ovar\left( \sum_{i=1}^k \psi_i(X_{S_i}) \right) \leq r  \sum_{i=1}^k \ovar\left( \psi_i(X_{S_i}) \right).\label{eq:MBvarianceDrop}
 \end{align}
\end{example}
\begin{proof}
Let $E_i = \operatorname{span}\{ (e_j)_{j\in S_i}\}$, where $e_j$ is the $j$th natural basis vector of $\mathbb{R}^n$.  The random vector $X = (X_1, \dots, X_n)$ takes values in $\mathbb{R}^n$, and has law which splits along $(E_i, E_i^{\perp})$ for all $i=1,\dots, k$.  Letting $\xi$ be the uniform measure on $\{E_1, \dots, E_k\}$, we have 
$$
 \sum_{i=1}^k P_{E_i} \xi(\{E_i\}) \leq  \frac{r}{k}\operatorname{id}_{\mathbb{R}^n}.
$$
Hence, the claim follows from \eqref{eq:BLvarianceDrop}.
\end{proof}
\begin{remark}
A weighted version of \eqref{eq:MBvarianceDrop} also appears in \cite{madiman2007generalized}, which can be obtained by letting $\xi$ be an arbitrary (non-uniform) distribution on $\{E_1, \dots, E_k\}$.
\end{remark}

The examples presented in this section have many interesting consequences of their own.  For example, applications of the Efron--Stein inequality are ubiquitous in statistics \cite{boucheron13},  and the Dembo--Kagan--Shepp inequality and the Madiman--Barron inequalities give simple proofs \cite{Courtade16, madiman2007generalized} of the monotonicity of entropy along the central limit theorem originally due to Arstein, Ball, Barthe, and Naor \cite{ABBN}, and its generalizations.    Also, \eqref{eq:fVarDrop2} in the special case where $E_{dep}=\{0\}$ was used by the author to establish generalized subadditivity estimates for Poincar\'e constants of convolution measures \cite{Courtade21}.

\section{Proof of Theorem \ref{thm:splittingwrtxi}}\label{sec:mainProof}

The proof of Theorem \ref{thm:splittingwrtxi} proceeds in two  steps.  First, we treat the case where $\xi$ is a discrete measure supported on a finite set of subspaces.  This avoids some of the subtleties that need to be resolved in the second step, where we only assume that $\mu$ splits with respect to a measure $\xi\in P(\mathbf{E})$.

To get started, we prove Proposition  \ref{prop:ChiZeroSet} to establish that the independent decomposition of $\mathbb{R}^n$ with respect to $\xi \in P(\mathbf{E})$ is well defined. 
\begin{proof}[Proof of Proposition \ref{prop:ChiZeroSet}]
Note that $0 \in \chi_{\xi}^{-1}(0)$.  If $\chi_{\xi}^{-1}(0)=\{0\}$, then the claim is trivial.  Hence, we can assume without loss of generality that $\{0\}\subsetneq \chi_{\xi}^{-1}(0)$.  It will suffice to show that $\chi_{\xi}^{-1}(0)$ is the  union of mutually orthogonal   subspaces; uniqueness follows immediately.  Toward this end, define  a binary relation $\sim$ on $V:=\chi_{\xi}^{-1}(0)\setminus\{0\}$ via:
$$
x \sim y ~~~\Leftrightarrow ~~~ \chi(x-y) = 0, ~~x,y\in V.  
$$
The relation $\sim$ is obviously reflexive and symmetric.   Now, note that $x\sim y$ implies
$$
x,y,(x-y) \in (E \cup E^{\perp}),~~~\xi-a.s. 
$$
Since $x,y\in V$ themselves, this is only possible if the vectors $x,y$ are both in $E$, or are both in $E^{\perp}$, $\xi$-a.s.  In other words, for $x,y\in V$, 
$$
x\sim y ~~\Leftrightarrow ~~x,y\in E ~\mbox{or}~x,y\in E^{\perp},~ \xi-a.s.
$$
Since $V$ excludes the zero-vector, $\sim$ is transitive, and is therefore an equivalence relation on $V$.  By similar logic, if $x,y\in V$ and $x\not\sim y$, then $x,y$ are orthogonal in $\mathbb{R}^n$.  As a result, $\sim$ partitions $V$ into at most $n$ equivalence classes, which we enumerate as $V = \cup_{\alpha} V_{\alpha}$, and $V_{\alpha} \perp V_{\alpha'}$ for distinct classes.   To conclude, we define $E_{\alpha} = \{0\}\cup V_{\alpha}$, which is a linear subspace of $\mathbb{R}^n$ by definitions.  \end{proof}

\subsection{Special case of Theorem \ref{thm:splittingwrtxi} when $\xi$ is a discrete measure}

Assume that $\xi$ is a discrete measure, supported on $k$ points $\{\tilde{E}_1, \dots, \tilde{E}_k\}\subset \mathbf{E}$.  With this assumption in place, the assertion that $\mu$ splits with respect to  $\xi$ simplifies to saying that $\mu$ splits along $(\tilde{E}_i , \tilde{E}_i^{\perp})$ for each $i=1,\dots,k$.    Moreover, in reference to the independent decomposition of $\mathbb{R}^n$ with respect to $\xi$, the independent subspaces $(E_{\alpha})_{\alpha}$ are explicitly identified   via 
$$
\cup_{\alpha}E_{\alpha} = \cap_{i=1}^k (\tilde{E}_i \cup \tilde{E}_i^{\perp}).
$$
That is, each independent subspace is of the form $E_{\alpha}= \cap_{i=1}^k \tilde{E}_i^{\alpha_i}$, where $\alpha_i \in \{o,\perp\}$ and $E^o :=E$.   In this context, a preliminary version of Theorem \ref{thm:splittingwrtxi} is the following statement. The goal of this subsection is to prove it.

\begin{theorem}\label{thm:rigidSplit}
 If $\mu\in P(E)$ splits along $(\tilde{E}_i, \tilde{E}_i^{\perp})$ for each $i=1, \dots, k$, then $\mu$ takes the product form
\begin{align}
 \mu = \left( \otimes_{\alpha} \mu_{E_{\alpha}}\right) \otimes \mu_{E_{dep}},  \label{eq:productForm}
 \end{align}
 with $\mu_{E_{dep}}\in P(E_{dep})$ a Gaussian measure. 
 \end{theorem}
 
 Theorem \ref{thm:rigidSplit} is a  probabilistic manifestation of Valdimarsson's core argument in \cite{valdimarsson08}.  In our proof, we'll point out the key steps that follow Valdimarsson's analysis. 
 
We first establish the case where there are no nontrivial independent subspaces, which has been stated already as Proposition \ref{prop:GaussianCharacterization}.  We will require the following two lemmas.

\begin{lemma}\label{lem:finiteMoments} Under the assumptions of Proposition \ref{prop:GaussianCharacterization},  $\int_{\mathbb{R}^n} \|x\|d\mu(x) <\infty$. 
\end{lemma}
\begin{proof}
First,  we claim that there is some $\theta>0$ such that 
$$
\max_{1\leq i\leq k} \min\{\|P_{\tilde{E}_i}x\|, \|P_{\tilde{E}_i^{\perp}}x\|\} \geq \theta \|x\|, ~~\forall x\in \mathbb{R}^n. 
$$
Suppose not.  By homogeneity, there exists a sequence $(u_n)_{n\geq 1}\subset \mathbb{R}^n$ with norm $\|u_n\|=1$ such that 
$$
\min\{\|P_{\tilde{E}_i}u_n\|, \|P_{\tilde{E}_i^{\perp}}u_n\|\} \to 0, ~~~\forall i=1,\dots, k. 
$$
By continuity and compactness, there is a unit vector $u\in \mathbb{R}^n$ such that 
$$
\min\{\|P_{\tilde{E}_i}u\|, \|P_{\tilde{E}_i^{\perp}}u\|\} =0.
$$
Hence $0 \neq u \in \cap_{i=1}^k (\tilde{E}_i \cup \tilde{E}_i^{\perp})$, which is the desired contradiction.   

Therefore,  for $X\sim \mu$ and all $t\geq 0$, 
\begin{align*}
\Pr\{\|X\|\geq t\} &\leq \Pr\left( \cup_{i=1}^k \{ \min\{\|P_{\tilde{E}_i}X\|, \|P_{\tilde{E}_i^{\perp}}X\|\} \geq \theta t\}\right)\\
&\leq \sum_{i=1}^k \Pr\left(  \{\|P_{\tilde{E}_i}X\| \geq \theta t\}\cap\{ \|P_{\tilde{E}_i^{\perp}}X\| \geq \theta t\}\right)\\
&=\sum_{i=1}^k \Pr \{\|P_{\tilde{E}_i}X\| \geq \theta t\} \Pr\{ \|P_{\tilde{E}_i^{\perp}}X\| \geq \theta t\} \\
&\leq k \left( \Pr \{\| X\| \geq \theta t\} \right)^2 .
\end{align*}
The key step is the third line, where we used the assumption that $\mu$ splits along $(\tilde{E}_i,\tilde{E}_i^{\perp})$ for each $i$.   By Bryc, Theorem 1.3.5, this self-majorization property of the complementary distribution function ensures that there are finite $M,\alpha, \beta >0$ (depending only on $k,\theta$) such that
$$
\Pr\{\|X\|\geq t\}  \leq M \exp(- \beta t^{\alpha}), ~~~t\geq 0. 
$$
This guarantees $\EE[\|X\|]<\infty$. 
\end{proof}
\begin{lemma}[{\cite[Proposition~2]{PW}}]\label{lem:linearGrowth}
Let $X$ be a random vector on $\mathbb{R}^n$ with $\EE\|X\|<\infty$, independent from $Z\sim N(0,\id_{\mathbb{R}^n})$.  The  random vector $X+Z$ has a non-vanishing density $f$ which has  derivatives of all orders, and satisfes
$$
\|\nabla \log f(x) \|\leq C(1+\|x\|), ~~~\forall x\in \mathbb{R}^n
$$
for some finite constant $C$ depending only on $\EE\|X\|$. 
\end{lemma}

\begin{proof}[Proof of Proposition \ref{prop:GaussianCharacterization}]
Since  the splitting property is stable under convolutions with isotropic Gaussians, we can assume $\mu$ is of the form $\mu_0 *\gamma$ for some probability measure $\mu_0\in P(\mathbb{R}^n)$.  By considering characteristic functions, showing $\mu$ is Gaussian establishes the same for $\mu_0$.

Let $f$ denote the density of $\mu$.  By Lemma \ref{lem:linearGrowth} and our regularizing assumption, $f$ is non-vanishing and has derivatives of all orders.  In combination with Lemma \ref{lem:finiteMoments}, we further have the linear growth estimate 
$$
\|\nabla \log f(x) \|\leq C(1+\|x\|), ~~~x\in \mathbb{R}^n
$$
for some $C$, depending only on the moments of $\mu$.   Now, just as in \cite{valdimarsson08}, this justifies taking the Fourier transform of  $s: x\mapsto \nabla \log f(x)$ as a tempered distribution;  we denote this Fourier transform by $\hat{s}$.   By the splitting property, $\hat{s}$ is supported on $(H_i + H_i^{\perp})$ for each $i$, where $H_i$ is the complexification of $\tilde{E}_i$.  Taking intersections over $i=1,\dots, k$, we find that   $\hat{s}$ is supported at the origin, and therefore $s = \nabla \log f$ is polynomial.  By the linear growth condition, we conclude $\nabla \log f$ must be affine, and therefore $\log f$ is quadratic.  Hence, $f$ is a Gaussian density, as desired.  \end{proof}

 We'll need the following lemma, which is a restatement of the classical fact that a bivariate function with vanishing mixed partial derivatives is separable in its variables.   It is taken for granted in \cite[p.~268]{valdimarsson08}, but finding no suitable reference, we opt to give its short proof here. 
\begin{lemma}\label{lem:SplittingVariables}
Let $V$ be a linear subspace of $\mathbb{R}^n$, and let  $\phi: \mathbb{R}^n \to \mathbb{R}$ be twice-differentiable.  The following are equivalent:
\begin{enumerate}[1)]
\item It holds that $\nabla^2 \phi(x) V \subset V$ for every $x\in \mathbb{R}^n$.
\item It holds that
$$
\phi(x) + \phi(0)= \phi(P_V x) + \phi(P_{V^{\perp}} x), ~~\forall x\in \mathbb{R}^n.
$$
\end{enumerate}
\end{lemma}
\begin{proof}
The direction $2) \Rightarrow 1)$ is trivial, so we only need to establish the reverse implication.  To that end, decompose $x = x_{\parallel} + x_{\perp}$, where $x_{\parallel} = P_V x$   and $x_{\perp}=P_{V^{\perp}} x$.  Let $v \in V$.  For convenience, we'll write $\langle \cdot, \cdot\rangle$ to denote the usual scalar product on $\mathbb{R}^n$.  For some $t \in (0,1)$ we claim 
\begin{align*}
\langle v ,   \nabla \phi(x) \rangle  
&\equiv \langle v ,   \nabla \phi(x_{\parallel},x_{\perp}) \rangle   \\
&
=  \langle v ,   \nabla \phi(x_{\parallel},0) \rangle   +  
\langle x_{\perp},\nabla  ( \langle v ,   \nabla \phi(x_{\parallel},t x_{\perp}) \rangle  ) \rangle  \\
&=  \langle v ,   \nabla \phi(x_{\parallel},0) \rangle     + \langle x_{\perp} , \nabla^2 \phi(x_{\parallel}, t x_{\perp})  v \rangle\\
&=  \langle v ,   \nabla \phi(x_{\parallel},0) \rangle   \\
&=\langle v , P_V \nabla \phi(P_V x) \rangle.
\end{align*}
Indeed, the  second line holds for some $t\in(0,1)$ by the mean value theorem, the third follows by definition of  the Hessian,   the fourth follows by orthogonality  because $V$ is an invariant subspace of $\nabla^2\phi$, and the last  because $\phi(x_{\parallel},0) \equiv \phi(P_V x)$,  $v\in V$ and $P_V$ is self-adjoint. 
Since $\nabla^2 \phi(x)$ is a symmetric matrix with invariant subspace $V$,  it follows that $V^{\perp}$  is also an invariant subspace of $\nabla^2\phi(x)$.  Hence, by a symmetric argument, we also have
$$
\langle v' ,   \nabla \phi(x) \rangle   = \langle v' , P_{V^{\perp}} \nabla \phi(P_{V^{\perp}}x) \rangle , ~~\forall v'\in V^{\perp}.
$$
Putting things together, this implies for all $v \in V, v'\in V^{\perp}$ that
\begin{align*}
\langle v+ v', \nabla \phi(x) \rangle  &= \langle v+ v', P_V \nabla \phi(P_V x)  +  P_{V^{\perp}} \nabla \phi(P_{V^{\perp}} x) \rangle \\
&=\langle v+ v',   \nabla (\phi\circ P_V)(x)  +  \nabla (\phi\circ P_{V^{\perp}})(x)  \rangle,
\end{align*}
where the second line followed from the chain rule and self-adjointness of the individual projections.   Since $v+v'$ can be taken to be an arbitrary vector in $E$, we must have $\phi   = \phi\circ P_V + \phi\circ P_{V^{\perp}}+C$ for some constant $C$.  Evaluating at $x=0$ reveals $C = - \phi(0)$. 
\end{proof}

\begin{proof}[Proof of Theorem \ref{thm:rigidSplit}]
Since the splitting property is stable under convolutions with isotropic Gaussians, we can assume $\mu$ is of the form $\mu_0 *\gamma$ for some probability measure $\mu_0\in P(\mathbb{R}^n)$.  Showing $\mu$ is of the required form shows the same for the initial data $\mu_0$.

Let $f$ denote the density of $\mu$.  By our regularizing assumption, $f$ is non-vanishing and has derivatives of all orders.  We now follow \cite[Proof of Lemma 13]{valdimarsson08}. Since $\log f$ is separable in the variables $P_{\tilde{E}_i} x$ and $P_{\tilde{E}^{\perp}_i} x$ for each $i$ by the splitting property, it follows from Lemma \ref{lem:SplittingVariables} that 
$$
\nabla^2 \log f(x) \tilde{E}^{\alpha_i}_i \subset \tilde{E}^{\alpha_i}_i, ~~~\forall x\in \mathbb{R}^n, \alpha_i\in\{o,\perp\}, i=1,\dots,k.
$$
Taking intersections, every independent subspace $E_{\alpha}$ satisfies
 $$
\nabla^2 \log f(x) E_{\alpha} \subset E_{\alpha}, ~~~\forall x\in \mathbb{R}^n,
$$
and therefore
 $$
\nabla^2 \log f(x) E_{ind} \subset E_{ind}, ~~~\forall x\in \mathbb{R}^n.
$$
By Lemma \ref{lem:SplittingVariables}, this implies that $\log f$ is of the form
$$
\log f(x) + \log f(0) = \log f(P_{E_{ind}}x) + \log f(P_{E_{dep}}x), ~~~\forall x\in \mathbb{R}^n.
$$
In other words, $\mu$ splits along $(E_{ind},E_{dep})$, and we can write $\mu = \mu_{E_{ind}} \otimes \mu_{E_{dep}}$.
 
A straightforward induction shows that $\mu_{E_{ind}}$ has product form $\otimes_{\alpha} \mu_{E_{\alpha}}$.  It remains to  show $\mu_{E_{dep}}$ is Gaussian.  Toward this end, since $E_{\alpha} \subset \tilde{E}^{\alpha_i}_i$ for each $i$ and suitable choice of ${\alpha_i}\in \{o,\perp\}$, it follows that $E_{dep}$ admits the orthogonal decomposition
$$
E_{dep} = (E_{dep}\cap \tilde{E}_i) \orthplus  (E_{dep}\cap \tilde{E}^{\perp}_i), ~~\mbox{for each $i=1,\dots,k$.}
$$
Hence, by the splitting property of $\mu$, we have that $\mu_{E_{dep}}$   splits along $(E_{dep}\cap \tilde{E}_i, E_{dep}\cap \tilde{E}^{\perp}_i)$ for each $i=1,\dots,k$, and 
$$
\cap_{i=1}^k ( (E_{dep}\cap \tilde{E}_i) \cup  (E_{dep}\cap \tilde{E}^{\perp}_i) ) = E_{dep}\cap \left( \cap_{i=1}^k (   \tilde{E}_i \cup    \tilde{E}^{\perp}_i) \right) = E_{dep}\cap E_{ind} = \{0\}.
$$
An application of Proposition \ref{prop:GaussianCharacterization} completes the proof.
\end{proof}

\subsection{Proof of Theorem \ref{thm:splittingwrtxi} in the general case}

\begin{proof}[Proof of Theorem \ref{thm:splittingwrtxi}] 
2)$\Rightarrow$1): 
If $E_{dep} = \mathbb{R}^n$, then $\mu$ is a Gaussian measure with covariance that splits with respect to $\xi$, and therefore $\mu$ splits along $(E , E^{\perp})$,   $\xi$-a.s. (see the remarks following Definition \ref{def:covSplits}).   So, assume $E_{dep}\subsetneq \mathbb{R}^n$, and let $B_{\alpha} = \{u_1, \dots, u_{\dim(E_{\alpha})}\}$ be a basis for a nonzero independent subspace $E_{\alpha}$.   By definition of $\chi_{\xi}$, for each $i=1,\dots, \dim(E_{\alpha})$, there exists a set $\mathbf{E}_{i,\alpha}\subset \mathbf{E}$ of full $\xi$-measure such that  
$$
u_i \in \cap_{E\in \mathbf{E}_{i,\alpha}}(E \cup E^{\perp}).
$$
Since $\dim(E_{\alpha})$ is finite, we have that $\mathbf{E}_{\alpha}:=\cap_{i=1}^{ \dim(E_{\alpha})}\mathbf{E}_{i,\alpha}$ is another set of full $\xi$-measure, and  $B_{\alpha} \subset \cap_{E\in \mathbf{E}_{\alpha}}(E \cup E^{\perp})$.  This implies
$$
E_{\alpha} \subset \cap_{E\in \mathbf{E}_{\alpha}}(E \cup E^{\perp}).
$$
Since there are at most $n$ distinct nonzero independent subspaces, the set $\mathbf{E}_{ind} := \cap_{\alpha} \mathbf{E}_{\alpha}$ also has  $\xi(\mathbf{E}_{ind})=1$, and has the property that exactly one of the following is true:
\begin{align}
P_E P_{E_{\alpha}} =P_{E_{\alpha}}~(\mbox{and~}P_{E_{\perp}} P_{E_{\alpha}} =0),~\mbox{or}~P_E P_{E_{\alpha}} =0 ~(\mbox{and~}P_{E_{\perp}} P_{E_{\alpha}} =P_{E_{\alpha}}). \label{eq:alternatives}
\end{align}
Finally, since $\mu_{E_{dep}}$ is Gaussian with covariance that splits with respect to $\xi$, there is a set $\mathbf{E}_{dep}\subset \mathbf{E}$ such that $\xi(\mathbf{E}_{dep})=1$ and $\mu_{E_{dep}}\otimes \delta_0$ splits along $(E, E^{\perp})$ for all $E\in \mathbf{E}_{dep}$, where $\delta_0$ is the centered dirac mass on $E_{ind}$.  This follows by Definition \ref{def:covSplits} and the following remarks.  

Therefore, collecting definitions, if $\mu$ admits the product form \eqref{eq:muProductXi}, then it splits along $(E,E^{\perp})$ for every $E\in \mathbf{E}_{ind}\cap \mathbf{E}_{dep}$.  Indeed, this is most easily seen by letting $X\sim \mu$ and writing
$$
P_E X = P_E P_{E_{dep}}X + \sum_{\alpha} P_E P_{E_{\alpha}}X, \mbox{~and~} P_{E^{\perp}} X = P_{E^{\perp}}  P_{E_{dep}}X + \sum_{\alpha} P_{E^{\perp}}  P_{E_{\alpha}}X.
$$
Assuming $E\in \mathbf{E}_{ind}\cap \mathbf{E}_{dep}$,  the Gaussian factors $P_E P_{E_{dep}}X$ and $P_{E^{\perp}} P_{E_{dep}}X$ are independent as just discussed.  Moreover, the factors $P_E P_{E_{\alpha}}X$ and $P_{E^{\perp}} P_{E_{\alpha}}X$ are independent by \eqref{eq:alternatives} for each $\alpha$.  Due to the assumed product form \eqref{eq:muProductXi}, we conclude that $\mu$ splits along $(E,E^{\perp})$ for every $E\in \mathbf{E}_{ind}\cap \mathbf{E}_{dep}$.   Since $\xi(\mathbf{E}_{ind}\cap \mathbf{E}_{dep})=1$, the claim follows. 

\medskip

\noindent 1)$\Rightarrow$2):  We can proceed in the same manner as above to identify a set $\mathbf{E}_{ind}$ of full $\xi$-measure such that 
$$
\cup_{\alpha}E_{\alpha}\subset \cap_{E\in \mathbf{E}_{ind}} (E \cup E^{\perp}).  
$$
Now, by hypothesis, there is a set  $\mathbf{E}_{split}$ of full $\xi$-measure such that $\mu$ splits along $(E,E^{\perp})$ for every $E \in \mathbf{E}_{split}$.   
Since $\mathbf{E}_{ind}\cap \mathbf{E}_{split}$ has full measure, we have 
$$
\chi_{\xi}(x) = \int_{\mathbf{E}_{ind} \cap \mathbf{E}_{split}} \min\{ \|P_E x\|,\|P_{E^{\perp}}x\|\} d\xi(E), ~~x\in \mathbb{R}^n,
$$
so that we conclude  
$$
\cup_{\alpha}E_{\alpha}\subset \cap_{E\in \mathbf{E}_{ind}} (E \cup E^{\perp}) \subset 
\cap_{E\in \mathbf{E}_{ind}\cap \mathbf{E}_{split} } (E \cup E^{\perp})
\subset \chi_{\xi}^{-1}(0).
$$
By virtue of Proposition \ref{prop:ChiZeroSet}, we conclude
\begin{align}
\cup_{\alpha}E_{\alpha} =  \cap_{E\in \mathbf{E}_{ind}\cap \mathbf{E}_{split} }  (E \cup E^{\perp}).\label{eq:Ealph}
\end{align}
Since everything is finite dimensional, there is a finite collection of subspaces $(\tilde{E}_i)_{i=1}^k \subset  \mathbf{E}_{ind}\cap \mathbf{E}_{split}$ such that
\begin{align}
\cup_{\alpha}E_{\alpha} = \cap_{i=1 }^k  (\tilde{E}_i \cup \tilde{E}_i^{\perp}) \label{extractedFinite}
\end{align}
Indeed, we can construct $(\tilde{E}_i)_{i=1}^k$ by a greedy algorithm, where $\tilde{E}_1$ is selected arbitrarily from $\mathbf{E}_{ind}\cap \mathbf{E}_{split}$, and $E_{m+1}$ for $m\geq 1$ is selected from $\mathbf{E}_{ind}\cap \mathbf{E}_{split}$ so that 
$$
\dim(\operatorname{sp}( \cap_{i=1 }^{m+1}  (\tilde{E}_i \cup \tilde{E}_i^{\perp}) )) < \dim(\operatorname{sp}( \cap_{i=1 }^{m}  (\tilde{E}_i \cup \tilde{E}_i^{\perp}) )), 
$$
unless if this is impossible, in which case we terminate the algorithm and set $k=m$. Until termination, this procedure reduces dimension by a positive integer number at each step, and we  always have
$$
\cup_{\alpha}E_{\alpha} \subset \cap_{i=1 }^m  (\tilde{E}_i \cup \tilde{E}_i^{\perp}), ~~\mbox{for each~$m\geq 1$}.
$$
So, in view of \eqref{eq:Ealph}, the procedure must terminate at some point with a collection $(\tilde{E}_i)_{i=1}^k$  satisfying \eqref{extractedFinite}. 

Since $(\tilde{E}_i)_{i=1}^k\subset \mathbf{E}_{split}$ by construction,   $\mu$ splits along $(\tilde{E}_i, \tilde{E}_i^{\perp})$ for each $i=1,\dots, k$.  So, an application of Theorem \ref{thm:rigidSplit} shows that $\mu$ admits the product form \eqref{eq:muProductXi}, with  $\mu_{E_{dep}}$ necessarily being a Gaussian measure.  Since $\mu$ splits along $(E,E^{\perp})$ $\xi$-a.s., and also along $(E_{dep}, E_{dep}^{\perp})$, it follows immediately that $\mu_{E_{dep}}$ has covariance that splits with respect to $\xi$. 

\medskip

\noindent 1) \& 2) $\Rightarrow$3):  This is trivial.   The assumption that $\mu$ splits along $(E, E^{\perp})$, $\xi$-a.s., ensures that \eqref{eq:MuIsMixture0} holds.  Since 2)  guarantees that $\mu_{E_{dep}}$ is Gaussian, it has finite second moments. 

\medskip

\noindent 3)$\Rightarrow$1):  Suppose $\mu$ is of the form \eqref{eq:MuIsMixture0}.  Without loss of generality, we can amplify the assumption that $\mu_{E_{dep}}$ has finite logarithmic moments to the  stronger assumption that $\mu$ has finite logarithmic moments.  This can be accomplished by contracting the measure $\mu$ along the directions in $E_{ind}$ (in fact,  we could assume without loss of generality that $\mu_{E_{ind}}$ is compactly supported).  By definition of the independent subspaces, both the mixture property \eqref{eq:MuIsMixture0} and the splitting property 1) are invariant to such transformations. 

Next, defining $\tilde{\mu} = \mu *   \gamma$, the fact that $\gamma = \gamma_E \otimes \gamma_{E^{\perp}}$ for every $E\in \mathbf{E}$  implies we also have
\begin{align}
\tilde{\mu} = \int_{\mathbf{E}} (\tilde{\mu}_E \otimes  \tilde{\mu}_{E^{\perp}}) d\xi(E). \label{eq:MuIsMixture1}
\end{align}
If $\tilde{\mu}$ splits along some $(E, E^{\perp})$, then so does $\mu$.   Therefore, it suffices to show that $\tilde{\mu}$ splits with respect to $\xi$.  Together with the assumption of finite logarithmic moments, this regularization via convolution implies finiteness of the Shannon entropy $|h(\tilde{\mu})| < \infty$, where Shannon entropy of a measure $\nu \ll \operatorname{Leb}$ is defined as
$$
h({\nu}):= -\int \log\left( \frac{d {\nu}}{dx}\right)  d {\nu} = -\int \frac{d\nu}{dx} \log \left( \frac{d {\nu}}{dx}\right)  d {x},
$$
provided the integral exists in the Lebesgue sense.    With this in hand, we write
\begin{align}
h(\tilde{\mu}) &\geq \int_{\mathbf{E}} h(\tilde{\mu}_E \otimes \tilde{\mu}_{E^{\perp}} ) d\xi(E) \label{eq:convexityOfD}\\
&=\int_{\mathbf{E}} \left( h(\tilde{\mu}_E  )  +   h( \tilde{\mu}_{E^{\perp}} ) \right) d\xi(E) \label{eq:GammaSplits} \\
&\geq   h(\tilde{\mu} ). \label{eq:applyBL}
\end{align}
In the above, \eqref{eq:convexityOfD} is (strict) concavity of the function $t\mapsto -t \log t$ on $t\geq 0$; \eqref{eq:GammaSplits} is because entropy is additive on product measures;   \eqref{eq:applyBL} is subadditivity of entropy.     Thus, we have equality throughout, and we conclude that Jensen's inequality  \eqref{eq:convexityOfD} must hold with equality since $h(\tilde{\mu})$ is finite.  This occurs only if  we have
$$
\frac{d (\tilde{\mu}_E \otimes \tilde{\mu}_{E^{\perp}}) }{dx} = \frac{d \tilde{\mu}  }{dx}~a.e., ~~\xi-a.s.
$$
This implies $\tilde{\mu} = \tilde{\mu}_E \otimes \tilde{\mu}_{E^{\perp}}$, $\xi$-a.s., as desired. \end{proof}

\appendix
 
 \section{Remarks on Question \ref{q:moments}}\label{app:Computations}

The following computation demonstrates simple instances where the answer to Question \ref{q:moments} is affirmative.  However, it is unclear how to extend the argument for this special case to the general setting. As we did in Section \ref{sec:Bernstein}, consider an ambient space $\mathbb{R}^n$ with dimension $n=2m$, and define
$$
E_1 = \{ (x,0) ;  x\in \mathbb{R}^m\}, ~~~E_2 = \{ (x,x); x\in \mathbb{R}^m\},
$$ 
where $0$  denotes the zero-vector in $\mathbb{R}^m$. For $\mu \in P(\mathbb{R}^n)$, define for convenience:
$$
\mu_1 :=\mu_{E_1}, ~~\mu_2 :=\mu_{E^{\perp}_1}, ~~\mu_+ :=\mu_{E_2}, ~~\mu_- :=\mu_{E_2^{\perp}},
$$
which are all probability measures on $\mathbb{R}^m$.  
Now, let $(X_1,X_2)\sim \mu_1\otimes \mu_2$, and $(X_+,X_-)\sim \mu_+\otimes \mu_-$.  If $\mu$ is representable as the mixture 
\begin{align}
\mu = p  \mu_1\otimes \mu_2 + (1-p) \mu_+\otimes \mu_1 \label{eq:simplemixture}
\end{align}
 for some $p\in (0,1)$, then computing the $E_1$ and $E_1^{\perp}$-marginals   from \eqref{eq:simplemixture} reveals 
$$
X_1 \overset{law}= \frac{X_+ + X_-}{\sqrt{2}}, ~~X_2 \overset{law}= \frac{X_+ - X_-}{\sqrt{2}}.
$$
 Similarly, computing the $E_2$ and $E_2^{\perp}$-marginals gives
 $$
X_+ \overset{law}= \frac{X_1 + X_2}{\sqrt{2}}, ~~X_- \overset{law}= \frac{X_1 - X_2}{\sqrt{2}}.
$$
 Thus, if $(X'_1,X_2')$ is an independent copy of $(X_1,X_2)$, we have
 $$
X_1 \overset{law}= \frac{X_1 + X_2 + X_1' - X_2'}{2}, ~~ 
X_2 \overset{law}= \frac{X_1 + X_2 - X_1' + X_2'}{2}
 $$
 Now, let $Z$ be equal in distribution to $X_1 + X_2 - X_1' - X_2'$, and let $(Z_i)_{i=1}^4$ be independent copies of $Z$.  By the above, we have
 $$
 Z \overset{law}= \frac{Z_1 + Z_2 + Z_3 + Z_4}{2}.
 $$
 Hence, $Z$ is a stable law, and noting the scale parameter involved, it must be Gaussian.  By an application of Cram\'er's decomposition theorem, $X_1$ and $X_2$ must be Gaussian.  In particular,  $\mu$ has finite logarithmic moment.   
  
Similar computations work when $E_2 =  \{ (a x, b x); x\in \mathbb{R}^m\}$ for nonzero $a,b\in \mathbb{R}$.  However, the calculations become unwieldy in more general settings, such as if $\mu$ is a nontrivial mixture of three distinct nontrivial products of marginals.

Given that Question \ref{q:moments} only requires us to establish finiteness of logarithmic moments (rather than the  stronger statement that $\mu$ is Gaussian), a more promising   approach may be to show that if $\mu$ admits mixture form \eqref{eq:MuIsMixture0}, then the tail probabilities  $\mu(\{ x : \|x\|^2\geq t\})$ can be nontrivially bounded to conclude finiteness of the logarithmic moment (e.g., as was done in the proof of Lemma \ref{lem:finiteMoments}).  For example, it suffices to show that there are $c,C<1$ and $t_0\geq 0$ such that
\begin{align}
\mu(\{ x : \|x\|^2 > t\}) \leq C \mu(\{ x : \|x\|^2 > c t\}), ~~~\forall t\geq t_0.\label{tailProbEstimate}
\end{align}
Indeed, for $X\sim \mu$, we have $\EE[\log(1+ \|X\|)]<\infty$ if and only if $\EE[\log(1+ \|X\|^2)]<\infty$.  Now, by the tail sum formula and a change of variables, 
$$
\EE[\log(1+ \|X\|^2)]= \int_{0}^{\infty} \frac{\mu(\{ x : \|x\|^2 > t\})}{1+t}dt.
$$
A simple induction shows that $t \mapsto  \frac{1}{1+t}\mu(\{ x : \|x\|^2 > t\})$ is integrable if \eqref{tailProbEstimate} holds.  
For now, however, Question \ref{q:moments} remains unanswered.

\end{document}